\documentclass[12pt]{amsart}

\usepackage{amsthm,amsmath,amssymb, amsaddr,datetime}
\usepackage[margin=1.25in]{geometry}
\usepackage{hyperref}

\usepackage{tikz}
\usepackage{makecell} 
\renewcommand{\deg}{\overline{d}} 

\newcommand{\Ecc}{\operatorname{Ecc}}
\newcommand{\ecc}{\operatorname{ecc}}
\newcommand{\diam}{\operatorname{diam}}
\newcommand{\rad}{\operatorname{rad}}
\newcommand{\degree}{\operatorname{deg}}

\theoremstyle{plain}
\newtheorem{theo}{Theorem}
\newtheorem{lemma}[theo]{Lemma}
\newtheorem{prop}[theo]{Proposition}
\newtheorem{obs}[theo]{Observation}

\theoremstyle{definition}
\newtheorem{defi}[theo]{Definition}

\theoremstyle{remark}
\newtheorem*{rem}{Remark}

\title[Eccentricity Sums in Trees]{Eccentricity Sums in Trees}

\author{Heather Smith\textsuperscript{1} \qquad \qquad L\'{a}szl\'{o} Sz\'{e}kely\textsuperscript{1} \qquad \qquad Hua Wang\textsuperscript{2}}
\address{\textsuperscript{1}
\small Department of Mathematics, University of South Carolina, Columbia, SC 29208, USA \\
\textsuperscript{2} Department of Mathematical Sciences, Georgia Southern University, Statesboro, GA 30460, USA
}
\thanks{Email:  {\tt smithhc5@math.sc.edu} (Heather Smith), {\tt szekely@math.sc.edu} (L\'aszl\'o Sz\'ekely), {\tt hwang@georgiasouthern.edu} (Hua Wang)}

\begin{document}

\begin{abstract}
  The eccentricity of a vertex, $\ecc_T(v) = \max_{u\in T} d_T(v,u)$, was one of the first, distance-based, tree invariants studied. The total eccentricity of a tree, $\Ecc(T)$, is the sum of eccentricities of its vertices. We determine extremal values and characterize extremal tree structures for the ratios $\Ecc(T)/\ecc_T(u)$, $\Ecc(T)/\ecc_T(v)$, $\ecc_T(u)/\ecc_T(v)$, and $\ecc_T(u)/\ecc_T(w)$ where $u,w$ are leaves of $T$ and $v$ is in the center of $T$. 
  In addition, we determine the tree structures that minimize and maximize total eccentricity among trees with a given degree sequence.  

\end{abstract}

\maketitle
\vspace{-0.2in}
\section{Introduction}

The \textit{eccentricity} of a vertex $v$ in a connected graph $G$ is defined  in terms of the distance function as 
$$ \ecc_G(v) := \max_{u \in V(G)} d(u,v) . $$
The radius of $G$, $\rad(G)$, is the minimum eccentricity while the diameter, $\diam(G)$, is the maximum. 
The center, $C(G)$, is the collection of vertices whose eccentricity is exactly $\rad(G)$. 

We focus our attention on trees, where the center has at most two vertices \cite{jordan} and the diameter is realized by a leaf. 
We also explore the \textit{total eccentricity} of a tree $T$, defined as the sum of the vertex eccentricities:  \[\Ecc(T) := \sum_{z \in V(T)} \ecc_T(z) .\]

For a fixed tree $T$ with $v \in C(T)$ and any $z\in V(T)$,
\[\min_{u\in L(T)}\frac{\Ecc(T)}{\ecc_T(u)} \leq \frac{\Ecc(T)}{\ecc_T(z)}\leq \frac{\Ecc(T)}{\ecc_T(v)}\]
where $L(T)$ denotes the leaf set of $T$. This motivates the study in Section 2 of the extremal values and structures for the following ratios where $u,w\in L(T)$ and $v\in C(T)$,
 \[\frac{\Ecc(T)}{\ecc_T(v)}, \quad \frac{\Ecc(T)}{\ecc_T(u)}, \quad \frac{\ecc_T(u)}{\ecc_T(v)}, \quad \text{ and } \quad\frac{\ecc_T(u)}{\ecc_T(w)}.\] 
The results are  analogous to similar studies on distance in \cite{barefoot} and on the number of subtrees in \cite{part1,part2}. 
As in those papers,  the behavior of ratios is more delicate than that of their numerators or denominators.

For a graph with $n$ vertices, the total eccentricity is $n$ times the average eccentricity.
 Dankelmann and Mukwembi \cite{mukw} gave sharp upper bounds on the average eccentricity of graphs in terms of independence number, chromatic number, domination number, as well as connected domination number. For trees with $n$ vertices, Dankelmann, Goddard, and Swart \cite{danke} showed that the path maximizes $\Ecc(T)$.  In Section 3, we prove that the star minimizes $\Ecc(T)$ among trees with a given order. Turning our attention to trees with a fixed degree sequence, we prove that the ``greedy" caterpillar maximizes $\Ecc(T)$ while the ``greedy" tree minimizes $\Ecc(T)$. This provides further information about the total eccentricity of ``greedy" trees across degree sequences. 

From here forward, we assume that $T$ is a tree with $n$ vertices. Given two vertices $a,b \in V(T)$, $P(a,b)$ will be the unique path between $a$ and $b$ in $T$.

\section{Extremal ratios}

In this section, we consistently use the letters $u,w$ to denote leaf vertices while $v$ is a center vertex. Before delving into ratios, the following observation from \cite{jordan} is given without proof, and will be used many times. The next observation is a simple calculation which will be useful in our proofs.

\begin{obs}\label{cl:gen}
The center, $C(T)$, contains at most 2 vertices. These vertices are located in the middle of a maximum length path, $P$. If $\{v\}=C(T)$, $v$ divides $P$ into two paths, each of length $\rad(T)$. If  $\{v,z\}=C(T)$, the removal of $vz\in E(T)$ will divide $P$ into two paths, each of length $\rad(T)-1$. 	
\end{obs}

\begin{obs}\label{path_ecc}
For any path $P$ with $y$ edges and $y+1$ vertices, 
\[\Ecc(P)= \sum_{z\in V(P)} \ecc_P(z) = \begin{cases}\frac{3}{4}y^2+y & \text{if $y$ is even}\\   \frac{3}{4}y^2+y+\frac{1}{4} & \text{if $y$ is odd.} \end{cases}\] 
\end{obs}

\subsection{On the extremal values of \texorpdfstring{$\frac{\Ecc(T)}{\ecc_T(v)}$}{Ecc(T)/ecc(v)} where \texorpdfstring{$v\in C(T)$}{v is in the center}}

\begin{theo}\label{2.1} Let $T$ be a tree with $n\geq 2$ vertices. For any $v\in C(T)$, we have 
\begin{equation*} \label{2n-1}
 \frac{\Ecc(T)}{\ecc_T(v)} \leq 2n-1. 
 \end{equation*} 
 For $n\geq 3$, equality holds if and only if $T$ is a star centered at $v$.
\end{theo}

\begin{proof}
Let $T$ be an arbitrary tree with $v\in C(T)$. It is known that for any tree $T$, $\diam(T) \leq 2\rad(T)$ and for any vertex $z\in V(T)$, $\rad(T)\leq \ecc_T(z) \leq \diam(T)$. Because $\ecc_T(v)=\rad(T)$, the bound in the theorem is proved as follows:
\[\Ecc(T) \leq \ecc_T(v) + (n-1) \diam(T)  \leq (2n-1) \rad(T).\]
Equality holds precisely when $T$ has $\ecc_T(z)=2\ecc_T(v)$ for all vertices $z\neq v$. Because the eccentricities of adjacent vertices differ by at most 1, $\ecc_T(v)=1$ and $\ecc_T(z)=2$ for all $z\neq x$ which is only true for the star. 
\end{proof}

\begin{theo}\label{2.2} Let $T$ be a tree with $n\geq 2$ vertices. Let $k$ and $i$ be nonnegative integers with $0\leq i\leq 2k$ and $n=k^2+i$. For any $v\in C(T)$, we have 
\begin{equation*} \label{alul}
 \frac{\Ecc(T)}{\ecc_T(v)} \geq \begin{cases}    n-3+2k+\frac{i}{k} & {\rm if\ } 0\leq i\leq k \cr n-3+2k+\frac{i+1}{k+1}   & {\rm if\ } k+1\leq i\leq 2k   .\end{cases} 
 \end{equation*} 
{For $n\geq 4$,} equality holds if and only if $T$ is a tree whose longest path has 2x vertices ($x=k$ in the first case and $x=k+1$ in the second) and each other vertex is adjacent to one of the two center vertices of this path.
For $i=k$, the two bounds agree and both values for $x$ provide an extremal tree.
\end{theo}

\begin{proof}
Let $T$ be a tree with $n\geq 3$ vertices and let $v\in C(T)$. If $T$ is a star then $\frac{\Ecc(T)}{\ecc_T(v)} = 2n-1$ which is strictly greater than the bounds in the theorem. 

For the remainder of the proof, we consider the case when $T$ is not a star. By Observation~\ref{cl:gen}, there is a maximum-length path $P:=P(u,w)$ with $v$ in the middle and $d(u,v)=\ecc_T(v)$. We now consider two cases, based upon the size of $C(T)$.

 If $C(T)=\{v\}$, then both $P(u,v)$ and $P(v,w)$ have length $\ecc_T(v)$. 
 Let $S$ be the non-empty set $\{w'\in L(T): w'\neq u \text{ and } d(v,w')= \ecc_T(v)\}$. 
 Create a new tree $F$ from $T$ by detaching each leaf $w'\in S$  and appending each one to $v$. This tree is different from $T$ because $T$ was not a star. For any $z\in V(T)$, $\ecc_T(z) \geq \ecc_F(z)$.
 Further, for each $w'\in S$, $\ecc_T(w')>\ecc_F(w')$. As a result, $\Ecc(T) > \Ecc(F)$.
 As for $v\in C(T)$, $\ecc_F(v)=\ecc_T(v)=d_T(u,v)$ because $u\not\in S$. 
 The length of the longest path in $F$ is one less than the length of the longest path in $T$ which implies $v\in C(F)$ and $|C(F)|=2$. 
Altogether, we see $\frac{\Ecc(T)}{\ecc_T(v)} > \frac{\Ecc(F)}{\ecc_F(v)}$.
Hence, to minimize $\frac{\Ecc(T)}{\ecc_T(v)}$, it suffices to consider those trees with two center vertices.
                          
Suppose $|C(T)|=2$ and let $x:=\ecc_T(v)$. Here, the path $P$ has length $2x-1$ and the vertices on $P$ realize their eccentricities along this path since it has maximum length. Explicitly calculating the eccentricities of the vertices on $P$, using Observation~\ref{path_ecc}, and lower bounding all other eccentricities by $x+1$, we have  
\[ \frac{\Ecc(T)}{\ecc_T(v)} \geq \frac{1}{x}\left( \left(3x^2-x\right)+ (n-2x)(x+1) \right)= x + (n-3) + \frac{n}{x} =: f(x)\]
Equality holds  if and only if each vertex not on $P$ is a neighbor of one of the center vertices of $P$, as in Fig.~\ref{fig:min_tv1}.

\begin{figure}[htbp]
\centering
    \begin{tikzpicture}[scale=1]
        \node[fill=black,circle,inner sep=1pt] (t1) at (0,0) {};
        \node[fill=black,circle,inner sep=1pt] (t2) at (1,0) {};
        \node[fill=black,circle,inner sep=1pt] (t3) at (2,0) {};
        \node[fill=black,circle,inner sep=1pt] (t4) at (3,0) {};
        \node[fill=black,circle,inner sep=1pt] (t5) at (4,0) {};
        \node[fill=black,circle,inner sep=1pt] (t6) at (5,0) {};
        \node[fill=black,circle,inner sep=1pt] (t7) at (6,0) {};
        \node[fill=black,circle,inner sep=1pt] (t8) at (7,0) {};
        
        \node[fill=black,circle,inner sep=1pt] (t15) at (2.65,-.5) {};
        \node[fill=black,circle,inner sep=1pt] (t17) at (3.35,-.5) {};
        \node[fill=black,circle,inner sep=1pt] (t18) at (3.65,-.5) {};
        \node[fill=black,circle,inner sep=1pt] (t20) at (4.35,-.5) {}; 
        
        \draw (t1)--(t2);
        \draw (t3)--(t6);
        \draw (t7)--(t8);
        \draw [dashed] (t2)--(t3);
        \draw [dashed] (t6)--(t7);
        
        \draw (t15)--(t4);
        \draw (t17)--(t4);
	 \draw (t18)--(t5);
        \draw (t20)--(t5);           
        
        \node at (3,-.4) {$\ldots$};
        \node at (4,-.4) {$\ldots$};
        \node at (4,.2) {$v$};
        \node at (7,.2) {$u$};
        \node at (0,.2) {$w$};
        \node at (6,-.35) {$\substack{\underbrace{\hspace{5.2 em}}\\{x-1 \text{ vertices}}}$};
        \node at (1,-.35) {$\substack{\underbrace{\hspace{5.2 em}}\\ {x-1 \text{ vertices}}}$};
        \node at (3.5,-.8) {$\substack{\underbrace{\hspace{5 em}}\\ {n-2x \text{ vertices}}}$};

        \end{tikzpicture}
\caption{A tree minimizing $\frac{\Ecc(T)}{\ecc_T(v)}$.}\label{fig:min_tv1}
\end{figure}
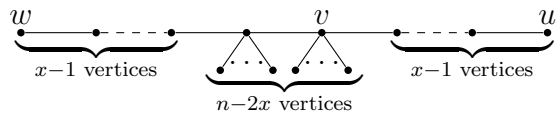 

To determine the value of $x$ which minimizes $f(x)$, we use the first derivative test. Because  
$ f'(x) = 1 - \frac{n}{x^2} $
is negative for $x< \sqrt{n}$ and positive for $x> \sqrt{n}$, the minimum of $f(x)$ is obtained when
$$ x \in\{ \left\lfloor \sqrt{n} \right\rfloor, \left\lceil \sqrt{n} \right\rceil\} \subseteq \{k,k+1\}. $$
Because $f(k+1)-f(k) = \frac{k-i}{k(k+1)}$, $f(k)\leq f(k+1)$ precisely when $i\geq k$ with equality when $i=k$, as stated in the theorem.
 \end{proof}

\subsection{On the extremal values of \texorpdfstring{$\frac{\Ecc(T)}{\ecc_T(u)}$}{Ecc(T)/ecc(u)} where \texorpdfstring{$u\in L(T)$}{u is a leaf}}

\begin{theo}
Let $T$ be a tree on {$n\geq 8$} vertices. Let $k$ and $i$ be integers with $0\leq i\leq 2k$ and $2n-1=k^2+i$. For any $u\in L(T)$, we have 

\begin{equation*} \label{leafub}
 \frac{\Ecc(T)}{\ecc_T(u)} \leq \begin{cases}    2n+1-2k-\frac{i}{k} & {\rm if\ } 0\leq i\leq k \cr 2n+1-2k-\frac{i+1}{k+1}   & {\rm if\ } k+1\leq i\leq 2k.\end{cases} 
 \end{equation*} 
Equality holds if and only if $T$ is a tree with longest path $P=z_1z_2\ldots z_{2x-1}$  ($x=k$ in the first case and $x=k+1$ in the second), leaf $u$ adjacent to $z_x$, and each other vertex adjacent to either $z_2$ or $z_{2x-2}$. For $i=k$, the two bounds agree and both values of $x$ will provide an extremal tree.
\end{theo}

\begin{proof}
Let $T$ be a tree with $n\geq 8$ vertices and $u\in L(T)$. If $T$ is a path, then $\frac{\Ecc(T)}{\ecc_T(u)} \leq \frac{3}{4}n+\frac{1}{2}$ which is strictly smaller than the bounds in the theorem.

For the remainder of the proof, we will suppose $T$ is not a path. Fix $P:=P(w,w')$ to be a maximum-length path in $T$. For any leaf $u\in L(T)$ different from $w$ and $w'$, $\frac{\Ecc(T)}{\ecc_T(w)}\leq \frac{\Ecc(T)}{\ecc_T(u)}$. Because we are interested in an upper bound, it suffices to consider leaves $u$ which are not on $P$. 

Let $u$ be a leaf of $T$ which is not on $P$ and let $x:=\ecc_T(u)$. There is a unique path from $u$ to the closest vertex on $P$, say $z$. Then $d(u,w)=d(u,z) + d(z,w)$ and $d(u,w')=d(u,z)+d(z,w')$. Since $d(u,w)$ and $d(u,w')$ are at most $x$ and $d(u,z)\geq 1$, we have $d(w,w')=d(w,z)+d(z,w') \leq 2x-2$. Because $P$ has maximum length, every vertex on $P$ realizes its eccentricity along $P$. Every vertex not on $P$ has eccentricity at most $2x-2$. This gives the upper bound
\begin{align*}
\frac{\Ecc(T)}{ \ecc_T(u)} &\leq 
\frac{1}{x} \left(x + \left(\frac{3}{4}(2x-2)^2 + (2x-2) \right)+ (n-2x)(2x-2) \right)
\\ & =  \frac{-x^2 + (2n+1)x - (2n-1)}{x}.\end{align*}
Equality is achieved precisely when $T$ is a tree with longest path $P$ on $2x-1$ vertices, $u$ is adjacent to the middle vertex  of $P$, and all other vertices have eccentricity $2x-2$. Such a tree $T$ is shown in Fig.~\ref{fig:max_tu1}. 

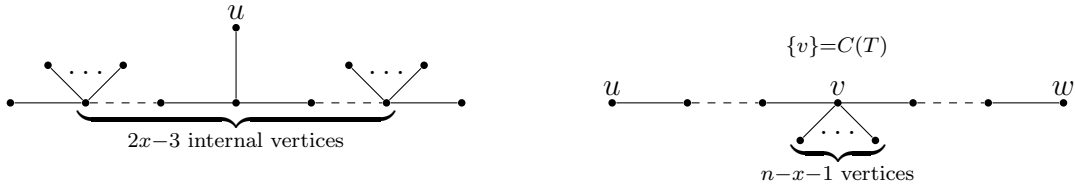
\begin{figure}[htbp]
\centering
    \begin{tikzpicture}[scale=1]
        \node[fill=black,circle,inner sep=1pt] (t1) at (0,0) {};
        \node[fill=black,circle,inner sep=1pt] (t2) at (1,0) {};
        \node[fill=black,circle,inner sep=1pt] (t3) at (2,0) {};
        \node[fill=black,circle,inner sep=1pt] (t4) at (3,0) {};
        \node[fill=black,circle,inner sep=1pt] (t5) at (4,0) {};
        \node[fill=black,circle,inner sep=1pt] (t6) at (5,0) {};
        \node[fill=black,circle,inner sep=1pt] (t7) at (6,0) {};
        \node[fill=black,circle,inner sep=1pt] (t8) at (3,1) {};
        
        \node[fill=black,circle,inner sep=1pt] (t15) at (.5,.5) {};
        \node[fill=black,circle,inner sep=1pt] (t17) at (1.5,.5) {};
        
        \node[fill=black,circle,inner sep=1pt] (t25) at (5.5,.5) {};
        \node[fill=black,circle,inner sep=1pt] (t27) at (4.5,.5) {};
        
        \draw (t1)--(t2);
        \draw (t3)--(t5);
        \draw (t6)--(t7);
        \draw (t4)--(t8);
        \draw [dashed] (t2)--(t3);
        \draw [dashed] (t5)--(t6);
        
        \draw (t15)--(t2);
        \draw (t17)--(t2);
        \draw (t25)--(t6);
        \draw (t27)--(t6);

        \node at (1,.4) {$\ldots$};
        \node at (5,.4) {$\ldots$};
        
        \node at (3,1.2) {$u$};

        \node at (3,-.35) {$\substack{\underbrace{\hspace{10.2 em}}\\ 2x-3 \text{ internal vertices}}$};
        
        \begin{scope}[shift={+(8,0)}]
        
        \node[fill=black,circle,inner sep=1pt] (t1) at (0,0) {};
        \node[fill=black,circle,inner sep=1pt] (t2) at (1,0) {};
        \node[fill=black,circle,inner sep=1pt] (t3) at (2,0) {};
        \node[fill=black,circle,inner sep=1pt] (t4) at (3,0) {};
        \node[fill=black,circle,inner sep=1pt] (t5) at (4,0) {};
        \node[fill=black,circle,inner sep=1pt] (t6) at (5,0) {};
        \node[fill=black,circle,inner sep=1pt] (t7) at (6,0) {};
        
        \node[fill=black,circle,inner sep=1pt] (t15) at (2.5,-.5) {};
        \node[fill=black,circle,inner sep=1pt] (t17) at (3.5,-.5) {};
        
        \draw (t1)--(t2);
        \draw (t3)--(t5);
        \draw (t6)--(t7);
        \draw [dashed] (t2)--(t3);
        \draw [dashed] (t5)--(t6);
        
        \draw (t15)--(t4);
        \draw (t17)--(t4);

        \node at (3,-.4) {$\ldots$};
        \node at (6,.2) {$w$};
        \node at (0,.2) {$u$};
        \node at (3,.2) {$v$};
      	\node at (3,.7) {$^{\{v\}=C(T)}$};
        \node at (3,-0.8) {$\substack{\underbrace{\hspace{3em}} \\ n-x-1 \text{ vertices} }$};
        \end{scope}
        \end{tikzpicture}
\caption{A tree (left) which maximizes $\frac{\Ecc(T)}{\ecc_T(u)}$ and a tree (right) which minimizes $\frac{\Ecc(T)}{\ecc_T(u)}$.}\label{fig:max_tu1}
\end{figure} 

It remains to determine the value of $x$ that will maximize $ \frac{\Ecc(T)}{\ecc_T(u)}$ for trees with the structure described above. The first derivative test shows that $f(x)$ is maximized when
$$ x \in\{\left\lfloor \sqrt{2n-1} \right\rfloor, \left\lceil \sqrt{2n-1} \right\rceil \}\subseteq \{k,k+1\}.$$ 
The larger of $f(k)$ and $f(k+1)$ gives the appropriate upper bound in \eqref{leafub}. In addition, we must require $2x\leq n$ in order to have a realizable tree. One can individually check that this is the case for $n\in \{8,9,\ldots, 12\}$. When $n\geq 13$, we have $k\geq 5$ in which case $0\leq  k^2-4k-3$ which implies $2x \leq 2(k+1)\leq n$. 
\end{proof}

\begin{theo}
Let $T$ be a tree of order $n\geq 5$. Let $k$ and $i$ be nonnegative integers with $0\leq i\leq 2k$ and $4n-4=k^2+i$. Then for any leaf $u$,
 \begin{equation} \label{leaflbeven}
 \frac{\Ecc(T)}{\ecc_T(u)} \geq 
 \begin{cases}  
 \frac{n-1}{2} + \frac{k}{2} + \frac{i}{4k} & \text{\textnormal{if $k$ is even }} \\ 
 \frac{n-1}{2} + \frac{k}{2} + \frac{i+1}{4(k+1)}  & \text{\textnormal{if $k$ is odd.}}
\end{cases}
\end{equation}
Equality holds if and only if $T$ is a tree with longest path $P$ of length $2x$ ($2x=k$ for the first case and $2x=k+1$ for the second) with all other vertices adjacent to the middle vertex of $P$ as shown in Fig.~\ref{fig:max_tu1}. When $i=k$, both bounds in \eqref{leaflbeven} give the same value and both give extremal structures.
\end{theo}

\begin{proof}
Let $T$ be a tree and $u \in L(T)$. Let $x:=\ecc_T(u)$ and choose $w\in L(T)$ so that $d(u,w)=x$. Let $P:=P(u,w)$. The vertices on $P$ have $\ecc_T(u) \geq \ecc_P(u)$. The eccentricity of any vertex not on $P$ is at least $1+\frac{x}{2}$ with equality if $x$ is even and these vertices are adjacent to the center vertex of $P$. This gives the following lower bound:
\begin{align*}
\frac{\Ecc(T)}{\ecc_T(u)}
&\geq \frac{1}{x}\left( \left(\frac{3}{4}x^2 + x\right) + (n-x-1)\left(1+\frac{x}{2}\right)\right)=:f(x)
\end{align*}
where equality holds when $P$ has even length and all vertices not on $P$ are adjacent to the center vertex of $P$ as in Fig.~\ref{fig:max_tu1}. 
Examination of $f'(x)$ shows that the ratio is minimized when 
$$ x \in\left\{ \left\lfloor \sqrt{4n-4} \right\rfloor,\left\lceil \sqrt{4n-4} \right\rceil \right\}\subseteq\left\{k,k+1\right\}. $$ 
We already established that the lower bound is tight for even $x$. For the universal lower bound,  we let $x=k$ if $k$ is even and $k+1$ otherwise. Both will yield a realizable tree because  $x\leq n-1$ for $n\geq 5$. It is also important to note that if $4n-4$ is a perfect square, then  $k=\lfloor \sqrt{4n-4} \rfloor=\lceil \sqrt{4n-4}\rceil=2\sqrt{n-1}$, an even value. The lower bounds in \eqref{leaflbeven} are exactly $f(k)$ and $f(k+1)$.
For thoroughness, it can be verified that $f(k)\leq f(k+2)$ and $f(k+1)\leq f(k-1)$, for $k>1$ to show that  our choice of the even integer nearest $\sqrt{4n-4}$ was correct for this concave up function. 
\end{proof}
 
 
 \subsection{On the extremal values of \texorpdfstring{$\frac{\ecc_T(u)}{\ecc_T(v)}$}{ecc(u)/ecc(v)} where \texorpdfstring{$u\in L(T)$}{u is a leaf} , \texorpdfstring{$v\in C(T)$}{v is in the center}}

\begin{theo} 
Let  $T$ be a tree on $n \geq 3$ vertices with $u\in L(T)$ and $v\in C(T)$. Then 
\begin{eqnarray*} 
\frac{\ecc_T(u)}{\ecc_T(v)} \leq 2,
\label{leaf_center_n4}
\end{eqnarray*}
where the upper bound  is tight for stars, even length paths, and more.
If, in addition, $n\geq 5$, then
\begin{eqnarray*}
 1+\frac{1}{\left\lfloor \frac{n-1}{2} \right\rfloor} \leq \frac{\ecc_T(u)}{\ecc_T(v)} . 
 \label{leaf_center_n5}
 \end{eqnarray*}
  Equality holds if and only if $T$ is one of the following trees: (1) For any $n\geq 5$, $T$ has a longest path $P$ on $n-1$ vertices with a single vertex $u$ adjacent to $v\in C(P)$. (2) For even $n\geq 5$, $T$ has a longest path $P$ on $n-2$ vertices with $u$ adjacent to $v\in C(P)$ and $w$ adjacent to any internal vertex of $P$. These structures are drawn in Fig.~\ref{fig:min_uv1}. 
\end{theo}

\begin{proof}
The upper bound of 2 follows from the facts that $\rad(T)= \ecc_T(v)$ and $\ecc_T(u)\leq \diam(T)\leq 2\rad(T)$.
This bound is tight for all trees whose maximum-length path has an odd number of vertices and $u$ a leave of one of these paths. 

Turning our attention to the lower bound, let $T$ be a tree with $n\geq 5$ vertices. 
We first show that it holds for paths. If $T$ is  a path, then 
\[\frac{\ecc_T(u)}{\ecc_T(v)} \geq \frac{n-1}{\frac{n}{2}}=  1+\frac{n-2}{n}  \geq 1+ \frac{1}{\left\lfloor \frac{n-1}{2} \right\rfloor}.\] 

For the remainder of the proof, we assume $T$ is not a path.  
Because $n\geq 5$, $\ecc_T(u) \geq \ecc_T(v) +1$ with equality exactly when $uv\in E(T)$. 
In addition, because $v\in C(T)$, Observation~\ref{cl:gen} guarantees a maximum-length path $P$ with $v$ in the middle. Because $T$ is not a path, $P$ has at most $n-1$ vertices and $\ecc_T(v) \leq \lceil\frac{n-2}{2}\rceil$. 
These two inequalities result in the desired bound.
\[\frac{\ecc_T(u)}{\ecc_T(v)} \geq 1 + \frac{1}{\ecc_T(v)} \geq 1 + \frac{1}{\lceil\frac{n-2}{2}\rceil}=1 + \frac{1}{\lfloor\frac{n-1}{2}\rfloor}.\]

Finally, let us analyze the trees $T$ for which equality holds. Because $P$ has at most $n-1$ vertices, we first examine the necessary and sufficient conditions to have $\ecc_T(v)=\lceil\frac{n-2}{2}\rceil$, based on the parity of $n$. If $n$ is odd, then $\ecc_T(v)=\frac{n-1}{2}$  if and only if $P$ has $n-1$ vertices.  For even $n$, $\ecc_T(v)=\frac{n-2}{2}$ if and only if $P$ has $n-1$ or $n-2$ vertices. 

Therefore, the bound in the theorem is tight exactly when $T$ is one of the following two trees which are drawn in Fig~\ref{fig:min_uv1}:  (1) $T$ is a tree with longest path $P$ on $n-1$ vertices and leaf $u$ adjacent to $v\in C(P)$. (2) For even $n$, $T$ is a tree with maximum-length path $P=z_1z_2\ldots z_{n-2}$ with $u$ adjacent to $v\in C(P)$ and $w$ adjacent to $z_i$ for some $i\in \{ 2,3, \ldots, n-3\}$.
\end{proof}

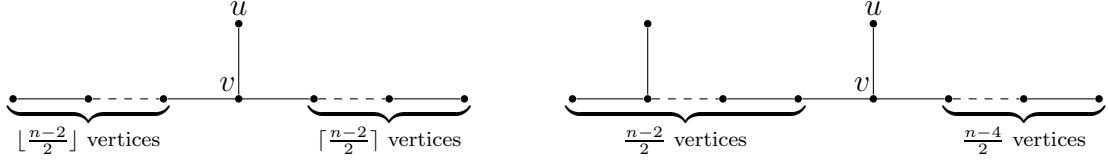
\begin{figure}[htbp]
\centering
\begin{tabular}{cc}
    \begin{tikzpicture}[scale=1]
        \node[fill=black,circle,inner sep=1pt] (t1) at (0,0) {};
        \node[fill=black,circle,inner sep=1pt] (t2) at (1,0) {};
        \node[fill=black,circle,inner sep=1pt] (t3) at (2,0) {};
        \node[fill=black,circle,inner sep=1pt] (t4) at (3,0) {};
        \node[fill=black,circle,inner sep=1pt] (t5) at (4,0) {};
        \node[fill=black,circle,inner sep=1pt] (t6) at (5,0) {};
        \node[fill=black,circle,inner sep=1pt] (t7) at (6,0) {};
        \node[fill=black,circle,inner sep=1pt] (t8) at (3,1) {};

        \draw (t1)--(t2);
        \draw (t3)--(t5);
        \draw (t6)--(t7);
        \draw (t4)--(t8);

        \draw [dashed] (t2)--(t3);
        \draw [dashed] (t5)--(t6);

        \node at (3,1.2) {$u$};
        \node at (2.85,.2) {$v$};
	 \node at (1,-.4) {$\substack{\underbrace{\hspace{5.2em}} \\ \lfloor\frac{n-2}{2} \rfloor\text{ vertices} }$};
        \node at (5,-.4) {$\substack{\underbrace{\hspace{5.2em}} \\ \lceil\frac{n-2}{2} \rceil\text{ vertices} }$};         
        \end{tikzpicture}
        & \hspace{.2in}
        \begin{tikzpicture}[scale=1]
        \node[fill=black,circle,inner sep=1pt] (t0) at (-1,0) {};
        \node[fill=black,circle,inner sep=1pt] (t1) at (0,0) {};
        \node[fill=black,circle,inner sep=1pt] (t2) at (1,0) {};
        \node[fill=black,circle,inner sep=1pt] (t3) at (2,0) {};
        \node[fill=black,circle,inner sep=1pt] (t4) at (3,0) {};
        \node[fill=black,circle,inner sep=1pt] (t5) at (4,0) {};
        \node[fill=black,circle,inner sep=1pt] (t6) at (5,0) {};
        \node[fill=black,circle,inner sep=1pt] (t7) at (6,0) {};
        \node[fill=black,circle,inner sep=1pt] (t8) at (3,1) {};
        \node[fill=black,circle,inner sep=1pt] (t9) at (0,1) {};
        
	 \draw (t0)--(t1);
	 \draw (t1)--(t9);
        \draw (t2)--(t3);
        \draw (t3)--(t5);
        \draw (t6)--(t7);
        \draw (t4)--(t8);

        \draw [dashed] (t1)--(t2);
        \draw [dashed] (t5)--(t6);

        \node at (3,1.2) {$u$};
        \node at (2.85,.2) {$v$};

         \node at (0.5,-.4) {$\substack{\underbrace{\hspace{7.7em}} \\  \frac{n-2}{2}\text{ vertices} }$}; 
        \node at (5,-.4) {$\substack{\underbrace{\hspace{5.2em}} \\ \frac{n-4}{2} \text{ vertices} }$}; 
        \end{tikzpicture}
        \end{tabular}
\caption{Trees which minimize $\frac{\ecc_T(u)}{\ecc_T(v)}$, the right one for even $n$ only.}\label{fig:min_uv1}
\end{figure} 

 
 \subsection{On the extremal values of \texorpdfstring{$\frac{\ecc_T(u)}{\ecc_T(w)}$}{ecc(u)/ecc(w)} where \texorpdfstring{$u,w\in  L(T)$}{u,w are leaves}}

First note that since the maximum and minimum values of $\frac{\ecc_T(u)}{\ecc_T(w)}$ are reciprocals of each other, we only consider the maximum.

\begin{theo}
Let $T$ be a tree with $n\geq 4$ vertices. For any $u,w\in L(T)$, we have
\begin{eqnarray*}
\frac{\ecc_T(u)}{\ecc_T(w)}\leq  2- \frac{2}{\lfloor \frac{n}{2}\rfloor}. 
\end{eqnarray*}
For even $n$, equality holds if and only if  $T$ is a tree with longest path $P=uz_2z_3\ldots z_{n-1}$ and leaf $w$ adjacent to $z_{n/2}$. For odd $n$, equality holds if and only if $T$ is a tree with longest path $P=uz_2z_3\ldots z_{n-2}$, leaf $w$ adjacent to $z_{(n-1)/2}$ and leaf $\omega$ adjacent to $z_i$ for some $i\in\{2, \ldots, n-3\}$. These constructions are drawn in Fig.~\ref{fig:max_uw1}.
\end{theo}

\begin{proof}
Let $T$ be a tree and let $u,w\in L(T)$. For the upper bound, it is reasonable to assume $\ecc_T(u)\geq \ecc_T(w)$. 
If $d(u,w)=\ecc_T(u)$, then $\frac{\ecc_T(u)}{\ecc_T(w)} =1$ which is strictly smaller than the bound in the theorem.

For the remainder of the proof, we focus on the case where $d(u,w)<\ecc_T(u)$. Choose $y\in L(T)$ so that $\ecc_T(u)=d(u,y)$ and let $P:=P(u,y)$. 
There is a unique path from $w$ to the nearest vertex, say $z$, on $P$.
Thus \[\ecc_T(w) \geq d(w,z) + \max\{d(z,u),d(z,y)\}\geq 1 + \left\lceil\frac{1}{2} \ecc_T(u)\right\rceil\] where equality holds if $d(w,z)=1$ and $|d(z,u)-d(z,y)| \leq 1$. We now consider two cases based on the parity of $\ecc_T(u)$.

First suppose $x:=\ecc_T(u)$ is odd. Let $S$ be the collection $\{y: d(u,y)=x\}$. Notice that $w$ is not in $S$ because $d(u,w)<x$. Now create a new tree $F$ from $T$ by detaching each $y\in S$ and reattaching each as a pendant vertex adjacent to the unique neighbor of $u$ in $T$. 
As a result, $\ecc_F(u) = x-1$, an even integer. By the above argument, $\ecc_T(w)\geq 1+ \lceil\frac{x}{2} \rceil = 1+\frac{1}{2}(x+1)$ while $\ecc_F(w) \geq 1+  \lceil\frac{1}{2} \ecc_F(u)\rceil  = 1+ \frac{1}{2}(x-1)$. As a result, we obtain tight upper bounds $\frac{\ecc_T(u)}{\ecc_T(w)}\leq \frac{x}{\frac{1}{2}(x+3)}$ and $\frac{\ecc_F(u)}{\ecc_F(w)}\leq \frac{x-1}{\frac{1}{2}(x+1)}$. The second gives the larger upper bound. Since we seek a tight universal upper bound for the ratio, it suffices to consider only trees with $u$ having even eccentricity. 

Assume $\ecc_T(u)$ is even. If $n$ is even, then $\ecc_T(u) \leq n-2$ because $w$ is not on $P$. However, we can tighten this to $\ecc_T(u) \leq n-3$ when $n$ is odd because of our assumption about the parity of $\ecc_T(u)$. In either case, $1+ \frac{1}{2}\ecc_T(u) \leq \lfloor \frac{n}{2} \rfloor$. This give the desired bound:
\[\frac{\ecc_T(u)}{\ecc_T(w)} \leq \frac{\ecc_T(u)}{1 + \frac{1}{2} \ecc_T(u)}
=2-\frac{2}{1+\frac{1}{2} \ecc_T(u)}
\leq 2 - \frac{2}{\lfloor \frac{n}{2} \rfloor} 
.\]

Finally, we characterize the trees $T$, based on the parity of $n$, for which equality holds. For even $n$, equality holds if and only if $T$ is a tree with longest path $P$ on $n-1$ vertices with leaf $w$ adjacent to the center of $P$. For odd $n$, equality holds if and only if $T$ is a tree with longest path $P$ on $n-2$ vertices with leaf $w$ adjacent to the center of $P$ and leaf $\omega$ adjacent to any internal vertex of $P$. This is exemplified by Fig.~\ref{fig:max_uw1},  with the additional leaf that occurs only for odd $n$ in gray. 
\end{proof}

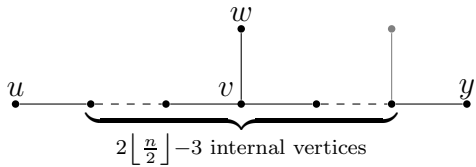
\begin{figure}[htbp]
\centering
    \begin{tikzpicture}[scale=1]
        \node[fill=black,circle,inner sep=1pt] (t1) at (0,0) {};
        \node[fill=black,circle,inner sep=1pt] (t2) at (1,0) {};
        \node[fill=black,circle,inner sep=1pt] (t3) at (2,0) {};
        \node[fill=black,circle,inner sep=1pt] (t4) at (3,0) {};
        \node[fill=black,circle,inner sep=1pt] (t5) at (4,0) {};
        \node[fill=black,circle,inner sep=1pt] (t6) at (5,0) {};
        \node[fill=black,circle,inner sep=1pt] (t7) at (6,0) {};
        \node[fill=black,circle,inner sep=1pt] (t8) at (3,1) {};
        \node[fill=gray,circle,inner sep=1pt] (t9) at (5,1) {};

        \draw (t1)--(t2);
        \draw (t3)--(t5);
        \draw (t6)--(t7);
        \draw (t4)--(t8);
        \draw [color=gray] (t6)--(t9);
        \draw [dashed] (t2)--(t3);
        \draw [dashed] (t5)--(t6);

        \node at (3,1.2) {$w$};
        \node at (2.8,.2) {$v$};
        \node at (0,.2) {$u$};
        \node at (6,.2) {$y$};

        \node at (3,-.45) {$\substack{\underbrace{\hspace{10 em}}\\ {2\left\lfloor \frac{n}{2} \right\rfloor-3 \text{ internal vertices}}}$};

        \end{tikzpicture}
\caption{A  tree maximizing  $\frac{\ecc_T(u)}{\ecc_T(w)}$.}\label{fig:max_uw1}
\end{figure}

Table~\ref{table:1} summarizes the results in section 2. Vertex labels appear as in the theorems. Specifically $v$ is always in $C(T)$ while each $u$ and $w$ are leaves of $T$. 

\begin{table}[ht]
\begin{center}
\begin{tabular} {!{\vrule width 1.3pt} m{2cm} |  m{3.6cm} | >{\centering\arraybackslash}m{4cm} !{\vrule width 1.3pt}}
\cline{2-3}
 \multicolumn{1}{c|}{}& \hspace{1.3cm}Bound & Extremal Tree\\ \noalign{\hrule height 1.3pt}
\Gape[10pt][0pt]{ $\displaystyle \frac{\Ecc(T)}{\ecc_T(v)}$ }& 
$\leq 2n-1$ & 
\begin{tikzpicture}[scale=0.5]
\foreach \x in {0,1,2,3,4,5,6} \draw [fill=black] (90-\x*45:1) circle (0.06);
\foreach \x in {0,1,2,3,4,5,6} \draw (0:0)--(90-\x*45:1);
\draw [fill=black] (0:0) circle (0.06);
\foreach \x in {0,1,2} \draw [fill=black] (122+13*\x:1) circle (0.015);
\end{tikzpicture}
\\
\hline
\Gape[10pt][0pt]{$\displaystyle\frac{\Ecc(T)}{\ecc_T(v)}$ }& 
$\geq n+2\sqrt{n}-O(1)$
 & 
    \begin{tikzpicture}[scale=.4, every node/.style={scale=.8}]
        \node[fill=black,circle,inner sep=1pt] (t1) at (0,0) {};
        \node[fill=black,circle,inner sep=1pt] (t2) at (1,0) {};
        \node[fill=black,circle,inner sep=1pt] (t3) at (2,0) {};
        \node[fill=black,circle,inner sep=1pt] (t4) at (3,0) {};
        \node[fill=black,circle,inner sep=1pt] (t5) at (4,0) {};
        \node[fill=black,circle,inner sep=1pt] (t6) at (5,0) {};
        \node[fill=black,circle,inner sep=1pt] (t7) at (6,0) {};
        \node[fill=black,circle,inner sep=1pt] (t8) at (7,0) {};
        
        \node[fill=black,circle,inner sep=1pt] (t15) at (2.65,-.5) {};
        \node[fill=black,circle,inner sep=1pt] (t17) at (3.35,-.5) {};
        \node[fill=black,circle,inner sep=1pt] (t18) at (3.65,-.5) {};
        \node[fill=black,circle,inner sep=1pt] (t20) at (4.35,-.5) {}; 
        
        \draw (t1)--(t2);
        \draw (t3)--(t6);
        \draw (t7)--(t8);
        \draw [dashed] (t2)--(t3);
        \draw [dashed] (t6)--(t7);
        
        \draw (t15)--(t4);
        \draw (t17)--(t4);
	 \draw (t18)--(t5);
        \draw (t20)--(t5);           
        
        \node at (3,-.4) {$\ldots$};
        \node at (4,-.4) {$\ldots$};
        \node at (4,.5) {$v$};
        \node at (3.5,-1.2) {$\underbrace{\hspace{2.7 em}}_{\approx n-2\sqrt{n}}$};
        \end{tikzpicture}
\\
\noalign{\hrule height 1.3pt}
\Gape[10pt][0pt]{$\displaystyle\frac{\Ecc(T)}{\ecc_T(u)}$ }
& 
$\leq 2n-2\sqrt{2n}+O(1)$
& 
   \begin{tikzpicture}[scale=.4, every node/.style={scale=.8}]
        \node[fill=black,circle,inner sep=1pt] (t1) at (0,0) {};
        \node[fill=black,circle,inner sep=1pt] (t2) at (1,0) {};
        \node[fill=black,circle,inner sep=1pt] (t3) at (2,0) {};
        \node[fill=black,circle,inner sep=1pt] (t4) at (3,0) {};
        \node[fill=black,circle,inner sep=1pt] (t5) at (4,0) {};
        \node[fill=black,circle,inner sep=1pt] (t6) at (5,0) {};
        \node[fill=black,circle,inner sep=1pt] (t7) at (6,0) {};
        \node[fill=black,circle,inner sep=1pt] (t8) at (3,1) {};
        
        \node[fill=black,circle,inner sep=1pt] (t15) at (.5,.5) {};
        \node[fill=black,circle,inner sep=1pt] (t17) at (1.5,.5) {};
        
        \node[fill=black,circle,inner sep=1pt] (t25) at (5.5,.5) {};
        \node[fill=black,circle,inner sep=1pt] (t27) at (4.5,.5) {};
        
        \draw (t1)--(t2);
        \draw (t3)--(t5);
        \draw (t6)--(t7);
        \draw (t4)--(t8);
        \draw [dashed] (t2)--(t3);
        \draw [dashed] (t5)--(t6);
        
        \draw (t15)--(t2);
        \draw (t17)--(t2);
        \draw (t25)--(t6);
        \draw (t27)--(t6);
                
        \node at (1,.4) {$\ldots$};
        \node at (5,.4) {$\ldots$};       
        \node at (3,1.4) {$u$};
        \node at (2.7,.3) {$v$};
        \node at (3,-.8) {$\substack{\underbrace{\hspace{5.9 em}}\\ {\approx 2\sqrt{2n}-3 \text{ internal}}}$}; 
        \end{tikzpicture}
\\ \hline
\Gape[10pt][0pt]{$\displaystyle \frac{\Ecc(T)}{\ecc_T(u)}$}
& 
$\geq \frac{1}{2}n + \sqrt{n} - O(1)$
& 
    \begin{tikzpicture}[scale=.4, every node/.style={scale=.8}]
        \node[fill=black,circle,inner sep=1pt] (t1) at (0,0) {};
        \node[fill=black,circle,inner sep=1pt] (t2) at (1,0) {};
        \node[fill=black,circle,inner sep=1pt] (t3) at (2,0) {};
        \node[fill=black,circle,inner sep=1pt] (t4) at (3,0) {};
        \node[fill=black,circle,inner sep=1pt] (t5) at (4,0) {};
        \node[fill=black,circle,inner sep=1pt] (t6) at (5,0) {};
        \node[fill=black,circle,inner sep=1pt] (t7) at (6,0) {};
        
        \node[fill=black,circle,inner sep=1pt] (t15) at (2.5,-.5) {};
        \node[fill=black,circle,inner sep=1pt] (t17) at (3.5,-.5) {};
        
        \draw (t1)--(t2);
        \draw (t3)--(t5);
        \draw (t6)--(t7);
        \draw [dashed] (t2)--(t3);
        \draw [dashed] (t5)--(t6);
        
        \draw (t15)--(t4);
        \draw (t17)--(t4);

        \node at (3,-.4) {$\ldots$};
        \node at (6,.4) {$w$};
        \node at (0,.4) {$u$};
        \node at (3,.4) {$v$};
      
        \node at (3,-1.2) {$\substack{\underbrace{\hspace{2em}} \\ \approx n-2\sqrt{n}-1}$};
        \end{tikzpicture}
\\ \noalign{\hrule height 1.3pt}
\Gape[10pt][0pt]{$\displaystyle\frac{\ecc_T(u)}{ \ecc_T(v)}$}
&
$\leq 2$
& 
Stars, even length paths with pendant edges, etc.
\\ \hline
\Gape[10pt][0pt]{$\displaystyle \frac{\ecc_T(u)}{ \ecc_T(v)}$}
&     
$\geq 1+\frac{2}{ n}+O(\frac{1}{n^2})$
& 
    \begin{tikzpicture}[scale=.4, every node/.style={scale=.8}]]
        \node[fill=black,circle,inner sep=1pt] (t1) at (0,0) {};
        \node[fill=black,circle,inner sep=1pt] (t2) at (1,0) {};
        \node[fill=black,circle,inner sep=1pt] (t3) at (2,0) {};
        \node[fill=black,circle,inner sep=1pt] (t4) at (3,0) {};
        \node[fill=black,circle,inner sep=1pt] (t5) at (4,0) {};
        \node[fill=black,circle,inner sep=1pt] (t6) at (5,0) {};
        \node[fill=black,circle,inner sep=1pt] (t7) at (6,0) {};
        \node[fill=black,circle,inner sep=1pt] (t8) at (3,1) {};
        \node[fill=gray,circle,inner sep=1pt] (t9) at (1,1) {};
        
        \draw (t1)--(t2);
        \draw (t3)--(t5);
        \draw (t6)--(t7);
        \draw (t4)--(t8);

        \draw [dashed] (t2)--(t3);
        \draw [dashed] (t5)--(t6);
     	 \draw [color=gray] (t9)--(t2);
	
        \node at (3,1.4) {$u$};
        \node at (2.7,.35) {$v$};

        \end{tikzpicture}
        
\\ \noalign{\hrule height 1.3pt}
\Gape[10pt][0pt]{$\displaystyle \frac{\ecc_T(u)}{ \ecc_T(w)}$}
&
$\leq 2-\frac{4}{ n} + O(\frac{1}{n^2})$
&
    \begin{tikzpicture}[scale=.4, every node/.style={scale=.8}]
        \node[fill=black,circle,inner sep=1pt] (t1) at (0,0) {};
        \node[fill=black,circle,inner sep=1pt] (t2) at (1,0) {};
        \node[fill=black,circle,inner sep=1pt] (t3) at (2,0) {};
        \node[fill=black,circle,inner sep=1pt] (t4) at (3,0) {};
        \node[fill=black,circle,inner sep=1pt] (t5) at (4,0) {};
        \node[fill=black,circle,inner sep=1pt] (t6) at (5,0) {};
        \node[fill=black,circle,inner sep=1pt] (t7) at (6,0) {};
        \node[fill=black,circle,inner sep=1pt] (t8) at (3,1) {};
        \node[fill=gray,circle,inner sep=1pt] (t9) at (5,1) {};
        
        \draw (t1)--(t2);
        \draw (t3)--(t5);
        \draw (t6)--(t7);
        \draw (t4)--(t8);
        \draw [color=gray, line width = 1pt] (t6)--(t9);
        \draw [dashed] (t2)--(t3);
        \draw [dashed] (t5)--(t6);
        
        \node at (3,1.4) {$w$};
        \node at (2.7,.4) {$v$};
        \node at (0,.4) {$u$};

        \end{tikzpicture}
        \\ \noalign{\hrule height 1.3pt}
        \end{tabular}
\end{center}
\caption{A summary of results for an arbitrary tree $T$ on $n$ vertices with $v\in C(T)$ and $u,w\in L(T)$.}


\label{table:1}
\end{table}

 \section{Extremal structures}
In this section, we fix a class of trees and find the ones in this class that maximize $\Ecc(T)$ and the ones that minimize $\Ecc(T)$.  First, we consider the trees on $n$ vertices. Then, we fix a degree sequence and search in the class of trees that realize this degree sequence. 

 \subsection{General trees}

For many indices, such as the sum of distances and the number of subtrees, the star and the path are extremal. Dankelmann, Goddard, and Swart \cite{danke} showed that the path maximizes $\Ecc(T)$ among trees with given order.  We show that the star minimizes $\Ecc(T)$ among trees with given order.

\begin{prop}
For any tree $T$ with $n>2$ vertices, 
$$ \Ecc(T) \geq 1 + 2(n-1) = 2n-1 $$
with equality if and only if $T$ is a star.
\end{prop}
\begin{proof}
Any tree with at least three vertices has at most one vertex which is adjacent to every other vertex (hence with eccentricity 1).  Thus we have
$$ \Ecc(T) \geq 1 + 2(n-1) = 2n-1. $$
Equality holds if and only if the single center vertex has eccentricity 1 and all other vertices have eccentricity 2. This characterizes the star.
\end{proof}
\subsection{Trees with given degree sequences}

Given a degree sequence, let $\mathcal{T}$ be the class of trees that realize this degree sequence. We determine which trees in $\mathcal{T}$ have total eccentricity equal to $\min_{T\in \mathcal{T}} \Ecc(T)$ or $\max_{T\in \mathcal{T}} \Ecc(T)$. We note that a  sequence $(d_1,d_2,\ldots, d_n)$ is the degree sequence for a tree if and only if $\sum_{i=1}^n d_i = 2(n-1)$ and each $d_i$ is a positive integer.
 \subsubsection{General Caterpillars}

Among all trees with a given degree sequence, the sum of distances is maximized by a caterpillar \cite{zhang2008} and the number of subtrees is minimized by a caterpillar \cite{drew, zhang_new}. However, completely characterizing the extremal caterpillar turns out to be a very difficult question in both cases. For the sum of distances, it is a quadratic assignment problem that is NP-hard in the ordinary sense and solvable in pseudo-polynomial time \cite{cela}.

\begin{defi}\cite{wang_ec}
For $n\geq 3$, let $\deg=(d_1, d_2, \ldots, d_n)$ be the non-decreasing degree sequence of a tree with $d_k>1$ and  $d_{k+1}=1$ for some $k\in [n-2]$. The {\it greedy caterpillar}, $T$, is constructed as follows: 
\begin{itemize}
\item Start with a path $P=z_1z_2\ldots z_k$. 
\item Let $\phi: \{z_i\}_{i=1}^k \rightarrow \{d_i\}_{i=1}^k$ be a one-to-one function such that, for each pair $i,j\in [k]$, if $\ecc_P(z_i)> \ecc_P(z_{j})$ then $\phi(z_i)\geq \phi(z_{j})$ . 
\item For each $i\in \{2,3,\ldots, k-1\}$, attach $\phi(z_i)-2$ pendant vertices to $z_i$. For $i\in \{1,k\}$, attach $\phi(z_i)-1$ pendant vertices to $z_i$. 
 \end{itemize}
 \label{gr_cat}
\end{defi}
Fig.~\ref{fig:gre_cat} gives two examples of greedy caterpillars and highlights the fact that greedy caterpillars are not unique.

\begin{figure}[htbp]
\centering
    \begin{tikzpicture}[scale=1]
        \node[fill=black,circle,inner sep=1pt] (t1) at (0,0) {};
        \node[fill=black,circle,inner sep=1pt] (t2) at (.7,0) {};
        \node[fill=black,circle,inner sep=1pt] (t4) at (1.5,0) {};       
        \node[fill=black,circle,inner sep=1pt] (t3) at (2.3,0) {};
        \node[fill=black,circle,inner sep=1pt] (t8) at (3,0) {};
        \node[fill=black,circle,inner sep=1pt] (t11) at (0,1) {};
        \node[fill=black,circle,inner sep=1pt] (t12) at (0,-1) {};
        \node[fill=black,circle,inner sep=1pt] (t81) at (3,-1) {};
        \node[fill=black,circle,inner sep=1pt] (t82) at (3,1) {};
        \node[fill=black,circle,inner sep=1pt] (t21) at (.4,-1) {};
        \node[fill=black,circle,inner sep=1pt] (t22) at (.7,-1) {};
        \node[fill=black,circle,inner sep=1pt] (t23) at (1,-1) {};
        \node[fill=black,circle,inner sep=1pt] (t31) at (2,-1) {};
        \node[fill=black,circle,inner sep=1pt] (t32) at (2.6,-1) {};
        \node[fill=black,circle,inner sep=1pt] (t41) at (1.2,-1) {};
        \node[fill=black,circle,inner sep=1pt] (t42) at (1.8,-1) {};

        \node[fill=black,circle,inner sep=1pt] (t5) at (-.5,-.5) {};
        \node[fill=black,circle,inner sep=1pt] (t6) at (-.5,0) {};
        \node[fill=black,circle,inner sep=1pt] (t7) at (-.5,.5) {};
        
        \node[fill=black,circle,inner sep=1pt] (t15) at (3.5,-.75) {};
        \node[fill=black,circle,inner sep=1pt] (t16) at (3.5,-.25) {};
        \node[fill=black,circle,inner sep=1pt] (t18) at (3.5,.25) {};
        \node[fill=black,circle,inner sep=1pt] (t17) at (3.5,.75) {};
        
        \draw (t1)--(t2)--(t4)--(t3)--(t8);
        \draw (t11)--(t1)--(t12);
        \draw (t21)--(t2)--(t22)--(t2)--(t23);
        \draw (t82)--(t8)--(t81);
        \draw (t5)--(t1);
        \draw (t6)--(t1);
        \draw (t7)--(t1);
        \draw (t31)--(t3)--(t32);
        \draw (t41)--(t4)--(t42);
        
        \draw (t15)--(t8);
        \draw (t16)--(t8);
        \draw (t17)--(t8);
        \draw (t18)--(t8);

       \begin{scope}[shift={+(7,0)}]
       \node[fill=black,circle,inner sep=1pt] (t1) at (0,0) {};
        \node[fill=black,circle,inner sep=1pt] (t2) at (.7,0) {};
        \node[fill=black,circle,inner sep=1pt] (t4) at (1.5,0) {};       
        \node[fill=black,circle,inner sep=1pt] (t3) at (2.3,0) {};
        \node[fill=black,circle,inner sep=1pt] (t8) at (3,0) {};
        \node[fill=black,circle,inner sep=1pt] (t11) at (0,1) {};
        \node[fill=black,circle,inner sep=1pt] (t12) at (0,-1) {};
        \node[fill=black,circle,inner sep=1pt] (t81) at (3,-1) {};
        \node[fill=black,circle,inner sep=1pt] (t82) at (3,1) {};
        \node[fill=black,circle,inner sep=1pt] (t21) at (.4,-1) {};
        \node[fill=black,circle,inner sep=1pt] (t22) at (.7,-1) {};
        \node[fill=black,circle,inner sep=1pt] (t23) at (1,-1) {};
        \node[fill=black,circle,inner sep=1pt] (t31) at (2,-1) {};
        \node[fill=black,circle,inner sep=1pt] (t32) at (2.6,-1) {};
        \node[fill=black,circle,inner sep=1pt] (t41) at (1.2,-1) {};
        \node[fill=black,circle,inner sep=1pt] (t42) at (1.8,-1) {};

        \node[fill=black,circle,inner sep=1pt] (t5) at (-.5,-.75) {};
        \node[fill=black,circle,inner sep=1pt] (t6) at (-.5,-.25) {};
        \node[fill=black,circle,inner sep=1pt] (t9) at (-.5,.25) {};
        \node[fill=black,circle,inner sep=1pt] (t7) at (-.5,.75) {};
        
        \node[fill=black,circle,inner sep=1pt] (t15) at (3.5,-.5) {};
        \node[fill=black,circle,inner sep=1pt] (t16) at (3.5,0) {};
        \node[fill=black,circle,inner sep=1pt] (t17) at (3.5,.5) {};
        
        \draw (t1)--(t2)--(t4)--(t3)--(t8);
        \draw (t11)--(t1)--(t12);
        \draw (t21)--(t2)--(t22)--(t2)--(t23);
        \draw (t82)--(t8)--(t81);
        \draw (t5)--(t1);
        \draw (t6)--(t1);
        \draw (t7)--(t1);
        \draw (t9)--(t1);
        \draw (t31)--(t3)--(t32);
        \draw (t41)--(t4)--(t42);
        
        \draw (t15)--(t8);
        \draw (t16)--(t8);
        \draw (t17)--(t8);

       \end{scope}
        \end{tikzpicture}
\caption{Non-isomorphic greedy caterpillars for degree sequence $ (7,6,5,4,4,1,\ldots,1) . $}\label{fig:gre_cat}
\end{figure}
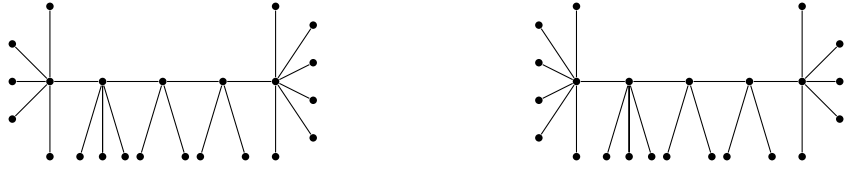

\begin{prop}
Among trees with a given tree degree sequence, the greedy caterpillar has the maximum total eccentricity.  
\end{prop}

\begin{proof}
Fix a degree sequence $\deg=(d_1, \ldots, d_n)$ which is written in the form described in Definition~\ref{gr_cat}. Let $\mathcal{T}$ be the collection of trees with degree sequence $\deg$. Let $T\in\mathcal{T}$ be a tree such that $\Ecc(T)=\max_{F\in\mathcal{T}} \Ecc(F)$. We first show that $T$ is a caterpillar. 

For contradiction, suppose $T$ is not a caterpillar. Let  $P_T(u,v)=uu_1u_2 \ldots u_{k}v$ be a longest path in $T$. Let $x\in[k]$ be the least integer such that $u_x$ has a nonleaf neighbor $w$ not on $P_T(u,v)$. Because $P$ is a maximum-length path, $x \neq 1$. Let $W$ be the component containing $w$ in $T-\{u_xw\}$. 

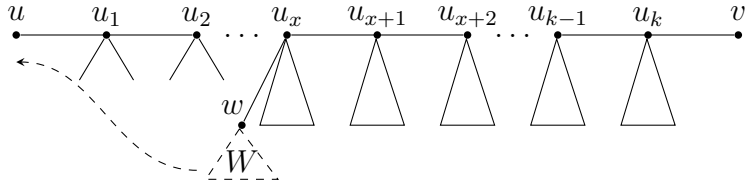
\begin{figure}[htbp]
\centering
    \begin{tikzpicture}[scale=1.2]
        \node[fill=black,circle,inner sep=1pt] (t1) at (0,0) {};
        \node[fill=black,circle,inner sep=1pt] (t2) at (1,0) {};
        \node[fill=black,circle,inner sep=1pt] (t3) at (2,0) {};       
        \node[fill=black,circle,inner sep=1pt] (t4) at (3,0) {};
        \node[fill=black,circle,inner sep=1pt] (t5) at (4,0) {};
        \node[fill=black,circle,inner sep=1pt] (t6) at (5,0) {};
        \node[fill=black,circle,inner sep=1pt] (t7) at (6,0) {};
        \node[fill=black,circle,inner sep=1pt] (t8) at (7,0) {};
        \node[fill=black,circle,inner sep=1pt] (t9) at (8,0) {};

        \node[fill=black,circle,inner sep=1pt] (w) at (2.5,-1) {};
        
        \draw (w)--(t4);
        \draw [dashed] (w)--(2.1, -1.6)--(2.9, -1.6)--cycle;
        \draw (2,-1.5) edge[out=180,in=0,->, >=stealth,dashed] (0,-.3);
        
        \draw (t4)--(2.7, -1)--(3.3, -1)--cycle;
        \draw (t5)--(3.7, -1)--(4.3, -1)--cycle;
        \draw (t6)--(4.7, -1)--(5.3, -1)--cycle;
        
        \draw (t7)--(5.7, -1)--(6.3, -1)--cycle;
        \draw (t8)--(6.7, -1)--(7.3, -1)--cycle;
        
        \draw (t3)--(1.7, -.5);
        \draw (t3)--(2.3, -.5);
        \draw (t2)--(.7, -.5);
        \draw (t2)--(1.3, -.5);
        \draw (t1)--(t3);
        \draw (t4)--(t6);
        \draw (t7)--(t9);
       
        \node at (2.5,0) {$\ldots$};
        \node at (5.5,0) {$\ldots$};
        \node at (0,.24) {$u$};
        \node at (1,.2) {$u_1$};
        \node at (2,.2) {$u_2$};
        \node at (3,.2) {$u_x$};
        \node at (4,.2) {$u_{x+1}$};
        \node at (5,.2) {$u_{x+2}$};
        \node at (6,.2) {$u_{k-1}$};
        \node at (7,.2) {$u_{k}$};
        \node at (8,.24) {$v$};
        
        \node at (2.4,-.8) {$w$};
        \node at (2.5,-1.4) {$W$};

        \end{tikzpicture}
\caption{Generating $T'$ from $T$.}\label{fig:gen_cat1}
\end{figure} 

Create a new tree $T'$ from $T$ by 
 replacing each edge of the form $zw$ in $W$ with the edge $zu$.
(Fig.~\ref{fig:gen_cat1}). Notice that $T$ and $T'$ have the same degree sequence. However, for any vertex $s \in \left(V(T) \setminus V(W)\right) \cup \{w\}$, $\ecc_{T'}(s) \geq \ecc_{T}(s)$ because $P_T(u,v)$ is a longest path in $T$. For any vertex $r\in V(W)-w$, we have
$$ \ecc_{T'}(r) = d(r,u) + d(u,v) > d(u,v) \geq \ecc_T(r) . $$
Thus $ \Ecc(T') > \Ecc(T)$, which contradicts the extremality of $T$. 

Since $T$ is a caterpillar with internal vertices forming path $P=u_1u_2\ldots u_k$, the eccentricity of any internal vertex is independent of the interval vertex degree assignments.
For any $i\in[k]$ and leaf $w$ adjacent to $u_i$,
$$ \ecc_T(w) = \max \{ k-i, i-1 \} + 2. $$
If $\phi: \{u_i\}_{i=1}^k\rightarrow \{d_i\}_{i=1}^k$ is a one-to-one function, then when $k$ is even,
\begin{align*}
\Ecc(T)=\sum_{i=1}^k \ecc_T(u_i) &+ (\phi(u_1)+\phi(u_k))(k+1) + (\phi(u_2)+\phi(u_{k-1}))(k) + \ldots \\ &+ \left(\phi\left(u_{k/2}\right) + \phi\left(u_{(k+2)/2}\right)\right) \left(k/2+2\right) . \end{align*}
In order to maximize the total eccentricity, for $i,j\in[k]$, if $j$ is closer to $k/2$ than $i$, then we should have $\phi(u_i) \geq \phi(u_j)$. It is a greedy caterpillar which achieves this. The case when $k$ is odd is similar.
\end{proof}

\subsubsection{Greedy trees and level-greedy trees}

In this subsection, each tree is rooted at a vertex. (While the root has no bearing on the total eccentricity, we use the added structure to direct our conversation.)
The height of a vertex is the distance to the root and the tree's height, $h(T)$, is the maximum of all vertex heights. We start with some definitions.

\begin{defi}\cite{nina}\label{def:leveldegree} In a rooted tree, the list of multisets $L_i$ of degrees of vertices
at height $i$, starting with $L_0$ containing the degree of the root vertex, is called the {\it level-degree sequence} of the rooted tree.
\end{defi}

Let  $|L_i|$ be the number of entries in $L_i$. It is easy to see that a list of multisets is the level degree sequence of a rooted tree if and only if (1) the multiset $\bigcup_i L_i$ is a tree degree sequence,  (2) $|L_0|=1$, and (3)
$\sum_{d\in L_0} d=|L_{1}|,$ and for all $i\geq 1$, $\sum_{d\in L_i} (d-1)=|L_{i+1}|.$ 

In a rooted tree, the \textit{down-degree} of the root is equal to its degree. The down degree of any other vertex is its degree minus one.

\begin{defi}\cite{nina}\label{def:greedy2} 
Given the level-degree sequence of a rooted tree, the {\it level-greedy rooted tree} for this level-degree sequence is built as follows: 
(1) For each $i\in [n]$, place $|L_i|$ vertices  in level $i$ and to each vertex, from left to right, assign a degree from $L_i$ in non-increasing order. (2) For $i\in [n-1]$, from left to right, join the next vertex in $L_i$ whose down-degree is $d$ to the first $d$ so far unconnected vertices on level $L_{i+1}$. Repeat for $i+1$.
\end{defi}

\begin{defi}\cite{wang_dm}\label{def_greedy}
Given a tree degree sequence $(d_1, d_2, \ldots, d_n)$ in non-increasing order, the {\it greedy tree} for this degree sequence is the level-greedy tree for the level-degree sequence that has $L_0=\{d_1\}$, $L_1=\{d_2, \ldots, d_{d_1+1}\}$ and for each $i>1$, \[|L_i|=\sum_{d\in L_{i-1}} (d-1)\] with every entry in $L_{i}$ at most as large as every entry in $L_{i-1}$.
\end{defi}

Fig.~\ref{greedy_pic} shows a greedy tree with degree sequence 
$ ( 4, 4, 4, 3,3,3,3,3,3,3,2,2, 1, \ldots , 1 ) . $

\begin{figure}[htbp]

\centering
    \begin{tikzpicture}[scale=0.4]
        \node[fill=black,circle,inner sep=1.5pt] (v) at (10,6) {}; 
        \node[fill=black,circle,inner sep=1.5pt] (v1) at (4,4) {};
        \node[fill=black,circle,inner sep=1.5pt] (v2) at (8,4) {};
        \node[fill=black,circle,inner sep=1.5pt] (v3) at (12,4) {};
        \node[fill=black,circle,inner sep=1.5pt] (v4) at (16,4) {};
        \node[fill=black,circle,inner sep=1.5pt] (v11) at (3,2) {};
        \node[fill=black,circle,inner sep=1.5pt] (v12) at (4,2) {};
        \node[fill=black,circle,inner sep=1.5pt] (v13) at (5,2) {};
        \node[fill=black,circle,inner sep=1.5pt] (v21) at (7,2) {};
        \node[fill=black,circle,inner sep=1.5pt] (v22) at (8,2) {};
        \node[fill=black,circle,inner sep=1.5pt] (v23) at (9,2) {};
        \node[fill=black,circle,inner sep=1.5pt] (v31) at (11,2) {};
        \node[fill=black,circle,inner sep=1.5pt] (v32) at (13,2) {};
        \node[fill=black,circle,inner sep=1.5pt] (v41) at (15,2) {};        
        \node[fill=black,circle,inner sep=1.5pt] (v42) at (17,2) {};         

        \node[fill=black,circle,inner sep=1.5pt] (v111) at (2.7,0) {};
        \node[fill=black,circle,inner sep=1.5pt] (v112) at (3.3,0) {};
        \node[fill=black,circle,inner sep=1.5pt] (v121) at (3.7,0) {};
        \node[fill=black,circle,inner sep=1.5pt] (v122) at (4.3,0) {};
        \node[fill=black,circle,inner sep=1.5pt] (v131) at (4.7,0) {};
        \node[fill=black,circle,inner sep=1.5pt] (v132) at (5.3,0) {};
        \node[fill=black,circle,inner sep=1.5pt] (v211) at (6.7,0) {};
        \node[fill=black,circle,inner sep=1.5pt] (v212) at (7.3,0) {};
        \node[fill=black,circle,inner sep=1.5pt] (v221) at (7.7,0) {};
        \node[fill=black,circle,inner sep=1.5pt] (v222) at (8.3,0) {};
        \node[fill=black,circle,inner sep=1.5pt] (v231) at (9,0) {};        
        \node[fill=black,circle,inner sep=1.5pt] (v311) at (11,0) {};

        \draw (v)--(v1);
        \draw (v)--(v2);
        \draw (v)--(v3);
        \draw (v)--(v4);
        \draw (v1)--(v11);
        \draw (v1)--(v12);
        \draw (v1)--(v13);
        \draw (v2)--(v21);
        \draw (v2)--(v22);
        \draw (v2)--(v23);
        \draw (v3)--(v31);
        \draw (v3)--(v32);
        \draw (v4)--(v41);
        \draw (v4)--(v42);
        \draw (v11)--(v111);
        \draw (v11)--(v112);
        \draw (v12)--(v121);
        \draw (v12)--(v122);
        \draw (v13)--(v131);
        \draw (v13)--(v132);
        \draw (v21)--(v211);
        \draw (v21)--(v212);
        \draw (v22)--(v221);
        \draw (v22)--(v222);
        \draw (v23)--(v231);
        \draw (v31)--(v311);
%
%
%

    \end{tikzpicture}

\caption{A greedy tree.}
\label{greedy_pic}
\end{figure}
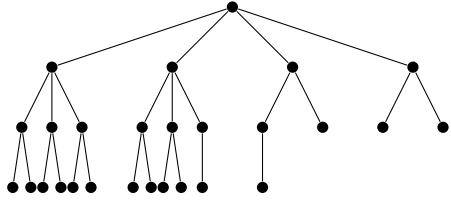

By definition, every greedy tree is level-greedy. 
However, Fig.~\ref{level_greedy_pic} shows a level-greedy tree that is not greedy. It has level degree sequence:
\[\{\{  3 \}, \,\{ 5, 3, 2 \}, \, 
\{ 3,3,3,2,2,1,1 \}, \, \{2,2,1,1,1,1,1,1\}, \,\{1,1\}\}.\]

\begin{figure}[htbp]

\centering
    \begin{tikzpicture}[scale=0.3]
        \node[fill=black,circle,inner sep=1.5pt] (v) at (10,8) {}; 

        \node[fill=black,circle,inner sep=1.5pt] (v1) at (5.5,6) {};
        \node[fill=black,circle,inner sep=1.5pt] (v2) at (12,6) {};
        \node[fill=black,circle,inner sep=1.5pt] (v3) at (15,6) {};

        \node[fill=black,circle,inner sep=1.5pt] (v11) at (1,4) {};
        \node[fill=black,circle,inner sep=1.5pt] (v12) at (4,4) {};
        \node[fill=black,circle,inner sep=1.5pt] (v13) at (7,4) {};
        \node[fill=black,circle,inner sep=1.5pt] (v14) at (9,4) {};
        \node[fill=black,circle,inner sep=1.5pt] (v21) at (11,4) {};
        \node[fill=black,circle,inner sep=1.5pt] (v22) at (13,4) {};
        \node[fill=black,circle,inner sep=1.5pt] (v31) at (15,4) {};

        \node[fill=black,circle,inner sep=1.5pt] (v111) at (0,2) {};
        \node[fill=black,circle,inner sep=1.5pt] (v112) at (2,2) {};
        \node[fill=black,circle,inner sep=1.5pt] (v121) at (3,2) {};
        \node[fill=black,circle,inner sep=1.5pt] (v122) at (5,2) {};
        \node[fill=black,circle,inner sep=1.5pt] (v131) at (6,2) {};
        \node[fill=black,circle,inner sep=1.5pt] (v132) at (8,2) {};
        \node[fill=black,circle,inner sep=1.5pt] (v141) at (9,2) {};
        \node[fill=black,circle,inner sep=1.5pt] (v211) at (11,2) {};

        \node[fill=black,circle,inner sep=1.5pt] (v1111) at (0,0) {};
        \node[fill=black,circle,inner sep=1.5pt] (v1121) at (2,0) {};

        \draw (v)--(v1);
        \draw (v)--(v2);
        \draw (v)--(v3);
        \draw (v1)--(v11);
        \draw (v1)--(v12);
        \draw (v1)--(v13);
        \draw (v1)--(v14);
        \draw (v2)--(v21);
        \draw (v2)--(v22);
        \draw (v3)--(v31);
        \draw (v11)--(v111);
        \draw (v11)--(v112);
        \draw (v12)--(v121);
        \draw (v12)--(v122);
        \draw (v13)--(v131);
        \draw (v13)--(v132);
        \draw (v14)--(v141);
        \draw (v21)--(v211);
        \draw (v111)--(v1111);
        \draw (v112)--(v1121);

%
%

    \end{tikzpicture}

\caption{A level-greedy tree.}
\label{level_greedy_pic}
\end{figure}
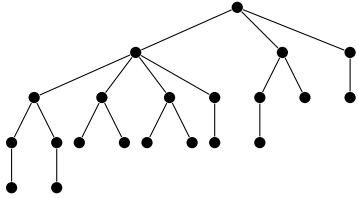

For a fixed degree sequence, greedy trees minimize the sum of distances \cite{nina, wang_dm, zhang2008} and maximize the number of subtrees \cite{eric, zhang2012}. We will show that they also minimize $\Ecc(T)$ among trees with a given degree sequence.

Here we provide some set-up for the proofs of the next two theorems. See Fig.~\ref{fig:prf_level} for an illustration.
Given a tree $T$ rooted at $v$, let $T_1$ be the subtree, rooted at child $v_1$ of $v$, containing some leaves of height $h:=h(T)$. Let $h':=h(T-T_1)$. Then for any vertex $u\in V(T - T_1)$ and any $w\in V(T_1)$ with $h_T(u)=h_T(w)=j$, then  
\begin{eqnarray} \ecc_T(u) = j + h \label{eccT2}, \end{eqnarray}
\begin{eqnarray} \ecc_T(w) = \max \{j + h', \ecc_{T_1}(w) \} \leq j + h \label{eccT1} \end{eqnarray}
where the first is only dependent on the height of $T$ and the second depends only on $h'$ and the structure of $T_1$.

\begin{figure}[htbp]
\centering
    \begin{tikzpicture}[scale=1]
        \node[fill=black,circle,inner sep=1.3pt] (t1) at (0,0) {};
        \node[fill=black,circle,inner sep=1.3pt] (t2) at (-4,-1) {};
        \node[fill=black,circle,inner sep=1pt] (t3) at (-4.5,-1.5) {};       
        \node[fill=black,circle,inner sep=1pt] (t4) at (1.5,-1.5) {};
        \node[fill=black,circle,inner sep=1pt] (t5) at (-5.5,-2.5) {};       
        \node[fill=black,circle,inner sep=1pt] (t6) at (2,-2) {};

        \draw (t1)--(t2);
        \draw (t2)--(-5.5,-2.5)--(-2.5,-2.5)--(t2);
        \draw (t1)--(-2,-2)--(2,-2)--(t1);
        
        \node at  (0,.2) {$v$};
        \node at  (-4.2,-.8) {$v_1$};
        \node at  (-5.6,-1.3) {$w$ (height $j$)};
        \node at  (2.6,-1.3) {$u$ (height $j$)};
        \node at  (-4,-2) {$T_1$};
        \node at  (0,-1.5) {$T-T_1$};
        \node at  (-5.5,-2.8) {(height $h$)};
        \node at  (2,-2.3) {(height $h'$)};
        \end{tikzpicture}
\caption{{A tree rooted at $v$ with $T_1$ a daughter subtree containing leaves of height $h$.}}\label{fig:prf_level}
\end{figure}
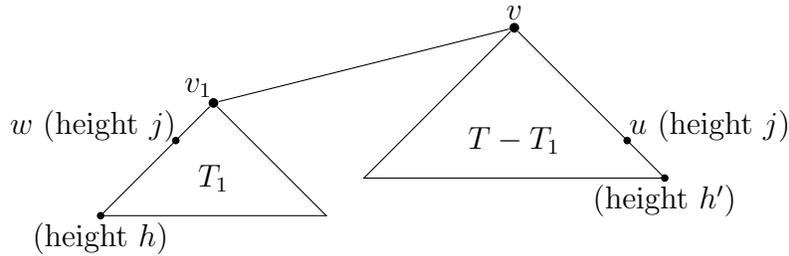

The following lemma implies that the level-greedy tree has the minimum total eccentricity among all rooted trees with a specified level-degree sequence.

\begin{lemma}\label{cl:level}
Let $\ell$ be a non-negative integer. Among the trees with a given level-degree sequence, the level-greedy tree maximizes the number of vertices having eccentricity at most $\ell$.
\end{lemma} 

\begin{proof} We proceed by induction on the number of vertices. The base case with one vertex is trivial. 

Fix $\ell>0$. Let $T$ be a rooted tree with the given level-degree sequence and the maximum number of vertices with eccentricity at most $\ell$. (i.e. $T$ is optimal.) 
 For vertices $w\in T_1$ and $u\in T-T_1$, both of height $j$, suppose for contradiction that $\degree(u)>\degree(w)$. Create a new tree $T'$ by moving $\degree(u)-\degree(w)$ children of $u$ and their descendants to adoptive parent $w$. This effectively switches the degrees of $u$ and $w$ while maintaining the level degree sequence. 
 
 While $\ecc_{T'}(u) = \ecc_T(u)$, notice that $h'$ did not increase and neither did $\ecc_T(w)$ for $w \in V(T_1)$. Since $ecc_{T'}(w) \leq \max\{j+h', \ecc_{T_1}(w)\} =\ecc_T(w)$, if strict inequality holds, then we have contradicted the optimality of $T$. Otherwise, $T'$ and $T$ are both optimal trees. In this case, we can repeat this shifting of degrees for pairs of vertices of height 1, followed by pairs of vertices of height 2, and so on until we either meet a contradiction or construct an optimal tree in which $\degree(u) \leq \degree(w)$ for all $w\in T_1$ and $u\in T-T_1$ of the same height. Assume that our optimal $T$ has this property.

Now we have a partition of the level-degree sequence for $T$ into level-degree sequences for $T-T_1$. By the inductive hypothesis, we may assume that both  $T_1$ and $T-T_1$ are level-greedy trees on their level-degree sequences. As a result, $T$ is a level-greedy tree.
\end{proof}

The next theorem also yields a stronger result than merely minimizing total eccentricity among trees with a given degree sequence. 
\begin{theo}
Let $\ell$ be a non-negative integer.  Among the trees with a given degree sequence, the greedy tree maximizes the number of vertices with eccentricity at most $\ell$.
\label{th:greedymin}
\end{theo}

\begin{proof}
Let $T$ be a tree with the given degree sequence with the maximum number of vertices with eccentricity at most $\ell$. (i.e. $T$ is optimal.) 
{Many times we will use the following claim: For two vertices $u$ and $v$ with $h(u) < \ell \leq h(v)$, it is preferable to assign degrees such that $\degree(u) \geq \degree(v)$ in order to maximize the number of vertices with height at most $\ell$.}

Find a longest path in $T$ and root $T$ at a center vertex  $v$ of that path. In $T-\{v\}$, let $T_1$ be the component with the leaf of greatest height. Let $v_1$ be the child of $v$ in $T_1$. By our choice of the root, if $h$ is the height of $T_1$, then the height of $T-T_1$ has height $h'\in \{h-1,h\}$. 
Now for any $w\in V(T_1)$ with $h_T(w)=j$, we have $\ecc_{T_1}(w)\leq (j-1)+(h-1)\leq j+h'-1$. In light of \eqref{eccT1},
\begin{equation*}\ecc_T(w)=\max\{j+h',\ecc_{T_1}(w)\}=j+h'.\label{Thm4.4ecc}\end{equation*}
For $w,x\in V(T_1)$, if $h_T(w) < h_T(x)$ then, by our earlier claim, $\ecc_T(w) < \ecc_T(x)$ which implies $\degree(w)\geq \degree(x)$ in $T$ because $T$ maximizes the number of vertices with small eccentricities. 

Vertices in $T-T_1$ with height $j$ have eccentricity $j+h$ by \eqref{eccT2}. So for $u,v\in V(T-T_1)$, when $h_T(u) < h_T(v)$, we can conclude $\degree(u)\geq \degree(v)$ in $T$. 

These observations establish the fact that either the root of $T-T_1$ or the root of $T_1$ has the largest degree in $T$. 

We now examine two cases based upon the value of $h'$.  When $h=h'$, we have  $\ecc_T(w)=j+h=\ecc_T(u)$ for any $w\in V(T_1)$, $u\in V(T-T_1)$ with $h_T(w)=h_T(u)=j$. Therefore, for $x,y\in V(T)$, if $h_T(x) < h_T(y)$, then $\degree(x)\geq \degree(y)$ in $T$. As an immediate consequence, the root of $T$ has the largest degree.

When $h'=h-1$, we may assume that the root of $T$ has the largest degree, for otherwise, we could reroot $T$ at $v_1$ which would not change the vertex eccentricities or the difference between $h$ and $h'$. 
Continuing in the setting with $h'=h-1$, for $w\in V(T_1)$ and $u,y\in V(T-T_1)$, if 
$h_T(w)=h_T(u)$, then $\ecc_T(w)=\ecc_T(u)-1$. So $\degree(w)\geq \degree(u)$ in $T$. However, if
$h_T(w)\geq h_T(y)+1$, then $\ecc_T(w)\geq \ecc_T(y)$. So we may assume $\degree(w)\leq \degree(y)$ in $T$. 

In both cases, we may assume that vertices of smaller height have larger degrees. Consequently, this  determines the level degree sequence of $T$. In fact, this is the level degree sequence for the greedy tree. The previous lemma asserts that we can assume $T$ is level-greedy. Therefore, $T$ is the greedy tree.
\end{proof}

\begin{rem}
Such extremal trees are not necessarily unique. In fact, the greedy tree gave a  much stronger restriction than what we needed, as stated in the theorem, while still not being the unique structure.
\end{rem}

\subsubsection{Greedy trees with different degree sequences}

As a final remark on greedy trees, given a collection of degree sequences, we order the corresponding greedy trees by their total eccentricity. The following observations, similar to previous works on other indices, yields many extremal results as immediate corollaries. For an example of such applications see \cite{zhang2012}.

\begin{defi}
Given two non-increasing sequences in $\mathbb{R}^n$, $\pi'=(d_1',\cdots,
d'_{n})$ and
 $\pi''=(d''_1, \cdots, d''_{n})$, $\pi''$ is said to {\it majorize}  $\pi'$, denoted $ \pi'\triangleleft \pi'' ,$ if for $k\in [n-1]$
 \begin{eqnarray*}
  \sum_{i=0}^{k}d'_i\le\sum_{i=0}^k d''_i \qquad \text{ and } \qquad \sum_{i=0}^{n}d'_i=\sum_{i=0}^{n}d''_i. \end{eqnarray*}
\end{defi}

\begin{lemma}\textup{\cite{wei1982}}\label{lem:wei} 
 Let $\pi'=(d'_1, \cdots d'_{n})$ and $\pi''=(d''_1, \cdots,
 d''_{n})$ be two non-increasing tree degree sequences. If
 $\pi'\triangleleft \pi'',$ then there exists a series of
 (non-increasing) tree degree sequences  $\pi^{(i)}=(d_1^{(i)}, \ldots, d_{n}^{(i)})$ for $1\leq i\leq m$ such that
 \[\pi' =\pi^{(1)}\triangleleft \pi^{(2)} \triangleleft \cdots \triangleleft \pi^{(m-1)} \triangleleft\pi^{(m)}= \pi''.\]
  In addition, each $\pi^{(i)}$ and $\pi^{(i+1)}$
differ at exactly two entries, say the $j$ and $k$ entries, $j<k$ where $d_j^{(i+1)} = d_j^{(i)} +1$ and $d_k^{(i+1)} = d_k^{(i)} -1$.
\end{lemma}

\begin{rem}
Lemma~\ref{lem:wei} is a {more refined version} of the original statement in \cite{wei1982}. In this process, each entry stays positive and the degree sequences remain non-increasing. Thereby, each obtained sequence is a tree degree sequence that is non-increasing without rearrangement. 
\end{rem}

\begin{theo}
Given two tree degree sequences $\pi'$ and $\pi''$ such that $\pi' \triangleleft \pi''$, \[\Ecc(T_{\pi'}^*)\geq \Ecc(T_{\pi''}^*)\] where $T_{\nu}^*$ is the greedy tree for degree sequence $\nu$. 
\end{theo}

\begin{proof}
According to Lemma~\ref{lem:wei}, it suffices to compare the total eccentricity of two greedy trees whose degree sequences differ in two entries, each by exactly 1, i.e., assume
$$ \pi'=(d'_1, \cdots d'_{n}) \triangleleft (d_1'', \cdots, d_{n}'')=\pi'' $$ 
 with $d''_j = d'_j +1$, $d''_k = d'_k -1$ for some $j<k$ and all other entries the same.

Let $u$ and $v$ be the vertices corresponding to $d'_j$ and $d'_k$ respectively and $w$ be a child of $v$ in $T_{\pi'}^*$ (Fig.~\ref{fig:pi}).
Construct $T_{\pi''}$ from $T_{\pi'}^*$ by removing the edge $vw$ and adding edge $uw$. Note that $T_{\pi''}$ has degree sequence $\pi''$ and by Theorem~\ref{th:greedymin}
$$ \Ecc(T_{\pi''}^*) \leq \Ecc(T_{\pi''}) . $$

\begin{figure}[htbp]
\centering
    \begin{tikzpicture}[scale=1.5]
        \node[fill=black,circle,inner sep=1pt] (t1) at (0,0) {};
        \node[fill=black,circle,inner sep=1pt] (t2) at (-.5,-.4) {};
        \node[fill=black,circle,inner sep=1pt] (t3) at (.5,-.4) {};
        \node[fill=black,circle,inner sep=1pt] (t4) at (-.8,-.8) {};
        \node[fill=black,circle,inner sep=1pt] (t5) at (-.2,-.8) {};
        \node[fill=black,circle,inner sep=1pt] (t6) at (.2,-.8) {};
        \node[fill=black,circle,inner sep=1pt] (t7) at (.8,-.8) {};
       
        \node[fill=black,circle,inner sep=1pt] (t20) at (-1.7,-1.2) {};
        \node[fill=black,circle,inner sep=1pt] (t21) at (-1.9,-1.2) {};
        \node[fill=black,circle,inner sep=1pt] (t22) at (-1.5,-1.2) {};
        \node[fill=black,circle,inner sep=1pt] (t23) at (-1.2,-1.2) {};
              
        \node[fill=black,circle,inner sep=1pt] (t13) at (1.5,-.4) {};
        \node[fill=black,circle,inner sep=1pt] (t14) at (-1.5,-.4) {};
        \node[fill=black,circle,inner sep=1pt] (t15) at (-1.8,-.8) {};
        \node[fill=black,circle,inner sep=1pt] (t16) at (-1.2,-.8) {};
        \node[fill=black,circle,inner sep=1pt] (t17) at (1.2,-.8) {};
        \node[fill=black,circle,inner sep=1pt] (t18) at (1.8,-.8) {};
        \node[fill=black,circle,inner sep=1pt] (t19) at (-1.5,-.8) {};
        
        \draw (t1)--(t13);
        \draw (t1)--(t14);
        \draw (t13)--(t17);
        \draw (t13)--(t18);
        \draw (t14)--(t15);
        \draw (t14)--(t16);
        \draw (t14)--(t19);
        
        \draw [line width=1mm] (t15)--(t20);
        \draw (t15)--(t21);
        \draw (t19)--(t22);
        \draw (t16)--(t23);

        \draw (t1)--(t2);
        \draw (t1)--(t3);
        \draw (t2)--(t4);
        \draw (t2)--(t5);
        \draw (t3)--(t6);
        \draw (t3)--(t7);
          
        \node[fill=black,circle,inner sep=1pt] (s1) at (4,0) {};
        \node[fill=black,circle,inner sep=1pt] (s2) at (3.5,-.4) {};
        \node[fill=black,circle,inner sep=1pt] (s3) at (4.5,-.4) {};
        \node[fill=black,circle,inner sep=1pt] (s4) at (3.2,-.8) {};
        \node[fill=black,circle,inner sep=1pt] (s5) at (3.8,-.8) {};
        \node[fill=black,circle,inner sep=1pt] (s6) at (4.2,-.8) {};
        \node[fill=black,circle,inner sep=1pt] (s7) at (4.8,-.8) {};
        
        \node[fill=black,circle,inner sep=1pt] (s20) at (3.5,-.8) {};
        \node[fill=black,circle,inner sep=1pt] (s21) at (2.8,-1.2) {};
        \node[fill=black,circle,inner sep=1pt] (s22) at (2.5,-1.2) {};
        \node[fill=black,circle,inner sep=1pt] (s23) at (2.2,-1.2) {};       
        
        \node[fill=black,circle,inner sep=1pt] (s13) at (5.5,-.4) {};
        \node[fill=black,circle,inner sep=1pt] (s14) at (2.5,-.4) {};
        \node[fill=black,circle,inner sep=1pt] (s15) at (2.2,-.8) {};
        \node[fill=black,circle,inner sep=1pt] (s16) at (2.8,-.8) {};
        \node[fill=black,circle,inner sep=1pt] (s17) at (5.2,-.8) {};
        \node[fill=black,circle,inner sep=1pt] (s18) at (5.8,-.8) {};
        \node[fill=black,circle,inner sep=1pt] (s19) at (2.5,-.8) {};
              
        \draw (s1)--(s13);
        \draw (s1)--(s14);
        \draw (s13)--(s17);
        \draw (s13)--(s18);
        \draw (s14)--(s15);
        \draw (s14)--(s16);
        \draw (s14)--(s19);
        
        \draw [line width=1mm] (s2)--(s20);
        \draw (s16)--(s21);
        \draw (s19)--(s22);
        \draw (s15)--(s23);

        \draw (s1)--(s2);
        \draw (s1)--(s3);
        \draw (s2)--(s4);
        \draw (s2)--(s5);
        \draw (s3)--(s6);
        \draw (s3)--(s7);
       
        \node at (-1.8,-.65) {$v$};
        \node at (-1.7,-1.35) {$w$};
                
        \node at (3.5,-.25) {$u$};
        \node at (3.5,-.95) {$w$};                
        
        \node at (0,-1.6) {$T_{\pi'}^*$};
        \node at (4,-1.6) {$T_{\pi ''}$};
     \end{tikzpicture}
\caption{ {\small $\pi = (4,4,3,3,3,3,2,2,1,\ldots,1)$ and 
$\pi' = (4,4,4,3,3,2,2,2,1,\ldots,1)$}.}\label{fig:pi}
\end{figure}
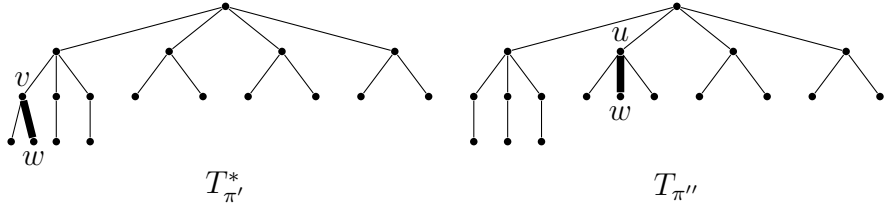

The height of any vertex in $T_{\pi''}$ is at most that of its counterpart in $T_{\pi'}^*$. An argument similar to that used in the proof of Lemma \ref{cl:level} shows
\begin{equation}\label{eq:com}
\Ecc(T_{\pi''}) \leq \Ecc(T_{\pi'}^*) .
\end{equation}
Hence
$ \Ecc(T_{\pi''}^*) \leq \Ecc(T_{\pi''}) \leq \Ecc(T_{\pi'}^*).$
\end{proof}

\begin{rem}
As in the proof of the extremality of greedy trees, equality holds more often in \eqref{eq:com} compared with its analogue for many other graph invariants. This also serves as some indication that $\Ecc(T)$ is not as strong of a graph invariant as compared to others in terms of characterizing the structures. 
\end{rem}

By comparing greedy trees with different degree sequences, the extremality of trees with respect to minimizing $\Ecc(.)$ under various restrictions easily follows. Consider, for example, trees with a given number of vertices and exactly $\ell$ leaves. The degree sequence of such a tree has exactly $\ell$ of 1's, where the degree sequence $(\ell, 2, \ldots , 2, 1, \ldots , 1)$ majorizes all other possible degree sequences. The corresponding greedy tree is a ``star-like" tree (a subdivision of star). Similarly, for trees with a given number of vertices and maximum degree $k$, the degree sequence $(k,k, \ldots , k, \ell , 1, \ldots 1)$ majorizes all other degree sequences with maximum degree $k$, where $\ell$ is the unique degree that is possibly between 1 and $k$. The corresponding greedy tree is called the ``extended good $k$-ary" tree. See for instance, \cite{bart, zhang2012} for details.

\section*{Acknowledgements}
L\'aszl\'o Sz\'ekely was supported in part by the  NSF DMS grant 1300547 and by the DARPA and AFOSR under the contract FA9550-12-1-0405.
Hua Wang was supported in part by the Simons Foundation (\#245307).
The authors wish to thank the referees for their ideas on how to shorten this paper and the proofs.

\bibliography{Bibliography}{}
 \bibliographystyle{plain} 

\end{document}